\documentclass{tran-l}
\usepackage{amsmath,amsthm,amssymb,mathtools}
\usepackage[breaklinks=true]{hyperref}
\allowdisplaybreaks

\theoremstyle{plain}
\newtheorem{theorem}{Theorem}[section]
\newtheorem{cor}[theorem]{Corollary}
\newtheorem{prop}[theorem]{Proposition}
\newtheorem{lemma}[theorem]{Lemma}

\newtheorem{utheorem}{\textrm{\textbf{Theorem}}}

\theoremstyle{definition}
\newtheorem{definition}[theorem]{Definition}
\newtheorem{remark}[theorem]{Remark}
\newtheorem{example}[theorem]{Example}

\numberwithin{equation}{section}

\newcommand{\bm}{\mathbf{m}}
\newcommand{\bn}{\mathbf{n}}
\newcommand{\bu}{\mathbf{u}}
\newcommand{\bv}{\mathbf{v}}

\newcommand\R{\mathbb{R}}
\newcommand\Z{\mathbb{Z}}
\newcommand\bp{\mathbb{P}}

\renewcommand{\geq}{\geqslant}
\renewcommand{\leq}{\leqslant}

\begin{document}


\title[Horn--Loewner positivity theorem; symmetric function
identities]{Smooth entrywise positivity preservers,\\
a Horn--Loewner master theorem,\\ and symmetric function identities}

\dedicatory{To Roger A.\ Horn and the memory of Charles Loewner, with
admiration}

\author{Apoorva Khare}
\address[A.~Khare]{Indian Institute of Science, Bangalore -- 560012,
India; and Analysis and Probability Research Group, Bangalore -- 560012,
India}
\email{\tt khare@iisc.ac.in}

\date{March 25, 2021, and, in revised form, September 21, 2021}



\thanks{This work was partially supported by Ramanujan Fellowship grant
SB/S2/RJN-121/2017, MATRICS grant MTR/2017/000295, and SwarnaJayanti
Fellowship grants SB/SJF/2019-20/14 and DST/SJF/MS/2019/3 from SERB and
DST (Govt.~of India), by grant F.510/25/CAS-II/2018(SAP-I) from UGC
(Govt.~of India), and by a Young Investigator Award from the Infosys
Foundation.}

\begin{abstract}
A special case of a fundamental result of Loewner and Horn
[\textit{Trans. Amer. Math. Soc.} 1969] says that given an integer $n
\geq 1$, if the entrywise application of a smooth function $f :
(0,\infty) \to \mathbb{R}$ preserves the set of $n \times n$ positive
semidefinite matrices with positive entries, then $f$ and its first $n-1$
derivatives are non-negative on $(0,\infty)$. In a recent joint work with
Belton--Guillot--Putinar [\textit{J.\ Eur.\ Math.\ Soc.}, in press], we
proved a stronger version, and used it to strengthen the
Schoenberg--Rudin characterization of dimension-free positivity
preservers [\textit{Duke Math.\ J.}\ 1942, 1959].

In recent works with Belton--Guillot--Putinar [\textit{Adv.\ Math.}\
2016] and with Tao [\textit{Amer.\ J.\ Math.}, in press] we used local,
real-analytic versions at the origin of the Horn--Loewner condition, and
discovered unexpected connections between entrywise polynomials
preserving positivity and Schur polynomials. In this paper,
we unify these two stories via a Master Theorem (Theorem~A) which
(i)~simultaneously unifies and extends all of the aforementioned
variants; and
(ii)~proves the positivity of the first $n$ nonzero Taylor coefficients
at individual points rather than on all of $(0,\infty)$.

A key step in the proof is a new determinantal / symmetric function
calculation (Theorem~B), which shows that Schur polynomials arise
naturally from considering arbitrary entrywise maps that are sufficiently
differentiable. Of independent interest may be the following application
to symmetric function theory: we extend the Schur function expansion of
Cauchy's (1841) determinant (whose matrix entries are geometric series $1
/ (1 - u_j v_k)$), as well as of a determinant of Frobenius [\textit{J.\
reine angew.\ Math.}\ 1882] (whose matrix entries are a sum of two
geometric series), to arbitrary power series, and over all commutative
rings.
\end{abstract}

\maketitle

\section{Introduction and main results}

Recently in~\cite{BGKP-fixeddim, BGKP-FPSAC, KT2, KT}, novel connections
were discovered between polynomials acting entrywise that preserve matrix
positivity and symmetric functions, specifically Schur polynomials. We
explore these connections in greater depth, and prove that Schur
polynomials emerge naturally from differentiating determinants involving
\textit{arbitrary} sufficiently differentiable functions. To the best of
our knowledge, this intriguing connection between analysis and symmetric
function theory is not recorded in the literature. This is our
Theorem~\ref{Phorn}, and as an application we prove the Master
Theorem~\ref{Tmain}, on functions acting entrywise that preserve positive
semidefinite matrices. We also apply it to provide a novel symmetric
function identity that subsumes modern results as well as classical ones
by Cauchy and Frobenius -- see Theorem~\ref{Tsymm}.

\subsection{Entrywise calculus and positivity: notation and history}

A real symmetric matrix $A_{n \times n}$ is positive semidefinite if all
its eigenvalues are non-negative real numbers -- equivalently, the
quadratic form $x \mapsto x^T A x, \ x \in \R^n$ is non-negative
definite, i.e., takes values in $[0,\infty)$.
Given an integer $n \geq 1$ and a domain $I \subset \R$, let $\bp_n(I)$
comprise the positive semidefinite $n \times n$ matrices with entries in
$I$. We work only with domains $I = (a,b)$ or $[a,b)$ for $0 \leq a < b
\leqslant \infty$.

A function $f : I \to \R$ acts \textit{entrywise} on a vector or a matrix
$A = (a_{jk}) \in I^{m \times n}$ for integers $m,n \geq 1$ via: $f[A] :=
(f(a_{jk}))$. If $f(x) = x^k$ is an integer power function, we write
$f[A] = A^{\circ k}$ for the entrywise power -- i.e., the $k$-fold
Schur/Hadamard product -- of $A$. Let ${\bf 1}_{m \times n}$ denote the
$m \times n$ matrix with all entries $1$.

A question of significant interest in the analysis literature throughout
the past century is to understand the entrywise
functions\footnote{We refer to polynomials -- and more generally,
functions -- that act entrywise on matrices as \textit{entrywise
polynomials/functions}, respectively.} that preserve positive
semidefiniteness, which we occasionally term \textit{positivity} in the
sequel. The first result in this area is the Schur product theorem
\cite{Schur1911}, which asserts that the set $\bp_n(I)$ is closed under
the entrywise product $A \circ B := (a_{jk} b_{jk})_{j,k=1}^n$ if $I$ is
closed under multiplication. Using that $\bp_n(\R)$ is a closed convex
cone, P\'olya and Szeg\"o \cite{polya-szego} observed the following
immediate consequence of the Schur product theorem: if $f(x) =
\sum_{k=0}^\infty c_k x^k$ is a convergent power series on $I$ and $c_k
\geq 0\ \forall k$, then $f[-]$ preserves positivity on $\bp_n(I)$. A
celebrated result of Schoenberg \cite{Schoenberg42}, subsequently
improved by Rudin \cite{Rudin59}, shows that there are no other functions
that preserve positivity in all dimensions for $I = (-1,1)$. These
results are motivated by and have connections to metric geometry,
positive definite functions, harmonic analysis, and analysis of measures
on Euclidean space and on tori.\footnote{See the surveys
\cite{BGKP-survey1,BGKP-survey2} for more on the classical and modern
interest in entrywise positivity preservers.}
Similar results have been shown for $I = (-\rho,\rho)$ or $(0,\rho)$ for
$0 < \rho \leq \infty$, as well as for complex domains -- in both one and
several variables -- in the years since Schoenberg's and Rudin's work.

Schoenberg's theorem has a challenging mathematical refinement that is
strongly motivated by modern-day applications in high-dimensional
covariance estimation. Namely: is it possible to classify the entrywise
positivity preservers in a \textit{fixed} dimension $n$? This problem was
resolved in 1979 by Vasudeva \cite{vasudeva79} for $n=2$, but remains
open for every $n \geq 3$. In this paper, we focus on a necessary
condition for this to happen.

\subsection{The Horn--Loewner theorem and its variants}

The focus of the present paper is a fundamental result on entrywise
preservers in fixed dimension $n \geq 3$. This result can be found in
Horn's 1967 thesis, and Horn attributes it to Loewner:\footnote{As the
author has also observed in the Stanford Library archives, the argument
in~\cite{horn} was outlined in a letter to Josephine Mitchell, written by
Loewner on October 24, 1967.}

\begin{theorem}[Necessary condition in fixed dimension, see
\cite{horn}]\label{Thorn}
Suppose $I = (0,\infty)$ and $f : I \to \R$ is continuous. Fix an integer
$n \geq 1$ and suppose $f[A] \in \bp_n(\R)\ \forall A \in \bp_n(I)$.
Then,
\[
f \in C^{n-3}(I), \qquad
f^{(k)}(x) \geq 0, \forall x \in I,\ 0 \leq k \leq n-3,
\]

\noindent and $f^{(n-3)}$ is a convex non-decreasing function on $I$. In
particular, if $f \in C^{n-1}(I)$, then $f^{(k)}(x) \geq 0$ for all $x
\in I, 0 \leq k \leq n-1$.
\end{theorem}

Theorem~\ref{Thorn} is important for several reasons:
\begin{enumerate}
\item To our knowledge, this 1967 result (together with its refinements,
which we will discuss) remains to this day the \textit{only} known
necessary condition for a function to be an entrywise positivity
preserver in a fixed dimension.

\item Theorem~\ref{Thorn} is sharp in the number of non-negative
derivatives of $f$ on $I$. We mention several such settings in
Example~\ref{Exbddunbdd}, providing in each case positivity preservers
that work in dimension $n$ but not $n+1$.

\item This fixed dimension result can be used to prove the dimension-free
version, i.e.~Schoenberg's theorem over $I = (0,\rho)$ for $0 < \rho \leq
\infty$; the proof uses Bernstein's theorem on absolutely monotonic
functions. In turn, the dimension-free version over $(0,\rho)$ (first
shown by Vasudeva~\cite{vasudeva79}) can be used to prove Schoenberg's
theorem over $I = (-\rho, \rho)$ by using less sophisticated machinery
compared to Schoenberg or Rudin's works. In fact this approach has proved
even more successful: in recent joint work \cite{BGKP-hankel}, we first
showed a stronger version of Theorem~\ref{Thorn}; then using it, a
strengthening of Schoenberg's theorem; and finally, multivariable
analogues of these results.
\end{enumerate}

In this paper, we are interested in strengthening Theorem~\ref{Thorn} in
multiple ways. We begin by stating several refinements proved in the
analysis literature -- our main result simultaneously extends all of
these variants.

The first strengthening begins with the observation that the argument in
the proof of Theorem~\ref{Thorn} is entirely local. With this in mind,
the domain of $f$ can be generalized to $(0,\rho)$ for any $0 < \rho \leq
\infty$. Moreover, the continuity assumption can be removed, in the
spirit of Rudin's strengthening \cite{Rudin59} of Schoenberg's theorem
\cite{Schoenberg42}. Finally, Theorem~\ref{Thorn} uses only a special
sub-family of matrices of rank at most $2$. Thus, in \cite{BGKP-hankel,
GKR-lowrank}, the following was shown (we state only the stronger of the
two results):

\begin{theorem}[{See \cite[Theorem 4.2]{BGKP-hankel}}]\label{Thankel}
Suppose that $0 < \rho \leq \infty$, $I = (0,\rho)$, and $f : I \to \R$.
Fix $u_0 \in (0,1)$ and an integer $n \geq 1$, and define $\bu := (1,
u_0, \dots, u_0^{n-1})^T$. Suppose that $f[A] \in \bp_2(\R)$ for all $A
\in \bp_2(I)$, and also that $f[A] \in \bp_n(\R)$ for all Hankel matrices
$A = a {\bf 1}_{n \times n} + t \bu \bu^T$, with $a \in \overline{I}$ and
$t \geq 0$ such that $a+t \in I$. Then the conclusions of
Theorem~\ref{Thorn} hold.
\end{theorem}

\noindent Theorem~\ref{Thankel} subsumes Theorem~\ref{Thorn} and the
aforementioned examples (see Example~\ref{Exbddunbdd}).

In this paper, we are interested in the case of smooth -- i.e.,
infinitely differentiable -- functions $f$, from which the general case
follows using a remarkable~1940 result by Boas--Widder~\cite{Boas-Widder}
(see Remark~\ref{RBWerratum}, which fixes a minor typo in their proof).
Thus, we henceforth work with smooth functions -- in which case the
assumption in Theorem~\ref{Thankel} concerning $\bp_2(I)$ is no longer
required; see \cite{BGKP-hankel} for details.

For smooth functions, one requires weaker assumptions, and reaches
stronger conclusions. Indeed, in ~\cite{BGKP-fixeddim} and~\cite{KT}, we
showed:

\begin{lemma}\label{Lhorn}
Let $n \geqslant 1$ and $0 < \rho \leqslant \infty$.
Suppose that $f(x) = \sum_{k \geq 0} c_k x^k$ is a convergent power
series on $I = [0,\rho)$ that is entrywise positivity
preserving\footnote{The results in \cite{KT} are stated for $I =
(0,\rho)$, but this is equivalent to using $[0,\rho)$ if $f$ is a power
series.} on rank-one matrices in $\bp_n(I)$.
Further assume that $c_{m'} < 0$ for some $m'$.
\begin{enumerate}
\item If $\rho < \infty$, then $c_m > 0$ for at least $n$ values of $m <
m'$. In particular, the first $n$ nonzero Maclaurin coefficients of $f$,
if they exist, must be positive.

\item If instead $\rho = \infty$, then $c_m > 0$ for at least $n$ values
of $m < m'$ and at least $n$ values of $m > m'$. In particular, if $f$ is
a polynomial, then the first $n$ nonzero coefficients and the last $n$
nonzero coefficients of $f$, if they exist, are all positive.
\end{enumerate}
\end{lemma}

We further showed in \cite{KT} that these conditions are sharp, in that
every other nonzero Maclaurin coefficient of $f$ can be negative. The
conclusions of Lemma~\ref{Lhorn} are stronger than those of
Theorem~\ref{Thorn} -- albeit at $a=0$ and not at $a>0$ -- and they also
cover settings not covered in Examples~\ref{Exbddunbdd}: the case of
\textit{all} polynomial preservers, not merely ones with the initial
Maclaurin coefficients of orders $0, 1, \dots, n-1$.

\subsection{The main theorems}

We now propose a stronger version of the Horn--Loewner theorem that
addresses the positivity of the first $n$ nonzero Taylor coefficients at
a given point; and moreover, one that unifies Theorem~\ref{Thorn} and
Lemma~\ref{Lhorn}, which do not imply one another. This is our first main
result; in the sequel, we work with $f$ smooth on $[a, a+\epsilon)$, and
we refer to $f^{(k)}(a^+)$ as $f^{(k)}(a)$.

\begin{utheorem}[Horn--Loewner master theorem]\label{Tmain}
Let $0 \leq a < \infty, \epsilon \in (0,\infty), I = [a, a+\epsilon)$,
and let $f : I \to \R$ be smooth. Fix integers $n \geq 1$ and $0 \leq p
\leq q \leq n$, with $p=0$ if $a=0$, and such that $f(x)$ has $q-p$
nonzero derivatives at $x=a$ of order at least $p$. Now suppose that
\[
p \leq m_p < m_{p+1} < \cdots < m_{q-1}
\]
are the lowest orders (above $p$) of the first $q-p$ nonzero derivatives
of $f(x)$ at $x=a$.

Also fix pairwise distinct scalars $u_1, \dots, u_n \in (0,1)$, and let
$\bu := (u_1, \dots, u_n)^T$.
If $f[ a {\bf 1}_{n \times n} + t \bu \bu^T ] \in \bp_n( \R)$ for all $t
\in [ 0, \epsilon)$, then $f^{(k)}(a)$ is non-negative for $0 \leq k \leq
m_{q-1}$.
\end{utheorem}

In particular, at $p=0$ one obtains the following corollary for
any $a \geq 0$, which strengthens the conclusions of the Horn--Loewner
Theorem~\ref{Thorn} and of Theorem~\ref{Thankel} for smooth functions:

\begin{cor}\label{Czeroder}
Suppose that $a, \epsilon, I, f, n, \bu$ are as in Theorem~\ref{Tmain}.
If $f[ a {\bf 1}_{n \times n} + t \bu \bu^T ] \in \bp_n( \R)$ for all $t
\in [ 0, \epsilon)$, then the first $n$ (or fewer) nonzero derivatives of
$f(x)$ at $x=a$ are positive.
\end{cor}

We make additional remarks about Theorem~\ref{Tmain} and its special
cases in Section~\ref{Sfurther}. For now, we note that it achieves
several objectives:
\begin{enumerate}
\item It unifies and further extends all of the above Horn--Loewner type
variants -- see Proposition~\ref{Pvariants} (hence, the term master
theorem).

\item Theorem~\ref{Tmain} yields more precise information than the
theorems stated above over $I = (0,\rho)$: about the derivatives of $f$
at each individual point $a > 0$ in the domain (rather than at all points
at once). The hypotheses employed are also local, which clarifies that
Theorem~\ref{Thorn} is not merely local but in fact a pointwise result.

\item By Corollary~\ref{Czeroder}, another strengthening accounts for the
zero derivatives at every point $a \geq 0$ -- and it is this
strengthening, whose sharpness (for $a=0$) we showed in \cite{KT} (see
the discussion following Lemma~\ref{Lhorn}).

\item An additional strengthening is when $a=0$ and $\rho < \infty$. In
this case, Theorem~\ref{Tmain} holds for all smooth functions, not merely
real analytic ones as in Lemma~\ref{Lhorn}(1).
\end{enumerate}

Theorem~\ref{Tmain} is proved using a novel connection between analysis
and symmetric function theory -- an explicit closed-form expression for
the derivatives of a certain determinant, which mixes calculus, matrix
algebra, and symmetric function theory. This is our second main theorem,
and it shows how Schur polynomials arise naturally in the entrywise
calculus, from considering (Cauchy-type) determinants for
\textit{arbitrary} sufficiently differentiable maps, not just polynomial
maps as in \cite{BGKP-fixeddim, KT} -- and at all $a \in \R$ (not just
$a=0$ as in \cite{BGKP-fixeddim, KT}):

\begin{utheorem}\label{Phorn}
Fix integers $n \geq 1$ and $0 \leq m_0 < m_1 < \cdots < m_{n-1}$, as
well as scalars $\epsilon > 0$ and $a \in \R$. Let $M := m_0 + \cdots +
m_{n-1}$ and let a function $f : [a, a+\epsilon) \to \R$ be $M$-times
differentiable at $a$ for some fixed $\epsilon > 0$. Fix vectors $\bu,
\bv \in \R^n$, and define $\Delta : [0,\epsilon') \to \R$ via:
\[
\Delta(t) := \det f[ a {\bf 1}_{n \times n} + t \bu \bv^T],
\]
for a sufficiently small $\epsilon' \in (0,\epsilon)$. Then,
\[
\Delta^{(M)}(0) = \sum_{\bm \vdash M} \binom{M}{m_0, m_1, \dots, m_{n-1}}
V(\bu) V(\bv) s_\bm(\bu) s_\bm(\bv) \prod_{k=0}^{n-1} f^{(m_k)}(a),
\]
where the first factor in the summand is the multinomial coefficient, and
we sum over all partitions $\bm = (m_{n-1}, \dots, m_0)$ of $M$, i.e., $M
= m_0 + \cdots + m_{n-1}$ and $m_{n-1} \geq \cdots \geq m_0$.

In particular, $\Delta(0) = \Delta'(0) = \cdots = \Delta^{(N-1)}(0) = 0$,
where $N = \binom{n}{2}$.
\end{utheorem}

In Section~\ref{S2}, we explain another application of
Theorem~\ref{Phorn}, to a symmetric function master identity -- see
Theorem~\ref{Tsymm} (in the $t$-adic topology; followed by
Corollary~\ref{Creal} in the real topology). We now explain the notation
required in Theorem~\ref{Phorn}. Our notation differs from that in the
literature~\cite{Macdonald}.

\begin{definition}[Schur polynomials, Vandermonde
determinants]\label{Dschur}\hfill
\begin{enumerate}
\item Given integers $m,N \geq 1$ and $0 \leq n'_0 \leq n'_1 \leq
\cdots \leq n'_{N-1}$, a \textit{column-strict Young tableau}, with shape
$\bn' := (n'_{N-1}, \dots, n'_0)$ and cell entries $1, 2, \dots, m$, is a
left-aligned two-dimensional rectangular array $T$ of cells, with $n'_0$
cells in the bottom row, $n'_1$ cells in the second lowest row, and so
on, such that:
\begin{itemize}
\item Each cell in $T$ has integer entry $j$ for some $j \in \{ 1, 2,
\dots, m \}$.
\item Entries weakly increase in each row, from left to right.
\item Entries strictly increase in each column, from top to bottom.
\end{itemize}

\item Given variables $u_1, u_2, \dots, u_m$ and a column-strict Young
tableau $T$, define its \textit{weight} to be
\[
{\rm wt}(T) := \prod_{j = 1}^m u_j^{f_j},
\]
where $f_j$ equals the number of cells in $T$ with entry $j$.

\item Given an increasing sequence of integers $0 \leq n_0 < \cdots <
n_{N-1}$, define the partitions/tuples
\[
\bn := (n_{N-1}, \dots, n_1, n_0), \qquad \bn_{\min} := (N-1, \dots, 1,
0),
\]
and the corresponding \textit{Schur polynomial} over $\bu := (u_1, u_2,
\dots, u_m)^T$ to be
\begin{equation}\label{Eschur}
s_\bn(\bu) := \sum_T {\rm wt}(T),
\end{equation}
where $T$ runs over all column-strict Young tableaux of shape $\bn' :=
\bn - \bn_{\min}$ with cell entries $1, 2, \dots, m$. By convention, we
set $s_\bn(\bu) := 0$ if $\bn$ does not have pairwise distinct
coordinates.

\item Given a vector $\bu = (u_1, \dots, u_m)^T$ with entries in a
commutative ring, we define its \emph{Vandermonde determinant} to be $1$
if $m=1$, and
\begin{equation}\label{Evandet}
V({\bf u}) := \prod_{1 \leq j < k \leq m} (u_k - u_j) = \det
\begin{pmatrix}
1 & u_1 & \cdots & u_1^{m-1} \\
1 & u_2 & \cdots & u_2^{m-1} \\
\vdots & \vdots & \ddots & \vdots \\
1 & u_n & \cdots & u_n^{m-1}
\end{pmatrix}, \qquad \text{if } m>1.
\end{equation}
\end{enumerate}
\end{definition}

The definition of Schur polynomials is due to Littlewood; notice that it
holds over the ground ring $\Z$, and hence over any commutative unital
ground ring. If $m=N$, then this definition is
equivalent~\cite{Macdonald} to Cauchy's definition of Schur polynomials:
\[
\det(\bu^{\circ n_0}\ |\ \bu^{\circ n_1} \ | \ \dots \ | \ \bu^{\circ
n_{N-1}})_{N \times N} = V(\bu) s_\bn(\bu).
\]
This is consistent with setting $s_\bn(\bu) = 0$ if $\bn$ has two equal
coordinates, since the left-hand matrix has two equal columns in that
case.

\begin{remark}\label{Rhorn}
From the proof of Theorem~\ref{Phorn} -- see the
expression~\eqref{Ederivative} and the discussion in the subsequent
paragraph -- it follows that if $f$ is smooth and has at most $n-1$
nonzero derivatives at $a$, then $\Delta^{(m)}(0) = 0\ \forall m \geq 0$.
\end{remark}

\begin{remark}
A curious point is that since one takes derivatives in
Theorem~\ref{Phorn}, one might expect one axis to deal with derivatives,
and the other to deal with the (anti-)symmetry in one set of variables --
but in fact, the final answer reveals the determinant's derivative to be
(anti-)symmetric in both sets of variables $\bu, \bv$.
\end{remark}

We close this section by returning to our starting point: the original
work of Loewner and Horn. We now explain how Theorem~\ref{Phorn} extends
a determinant computation by Loewner (in a~1967 letter, see footnote~3 --
and also found in Horn's 1967 thesis and paper \cite{horn}) in several
ways:
\begin{enumerate}
\item Loewner's computation was for $\bu = \bv$; Theorem~\ref{Phorn}
decouples the two vectors $\bu, \bv$.

\item Loewner showed Theorem~\ref{Phorn} only for $M \leq \binom{n}{2}$
(this is all that is required to obtain the original conclusions of the
Horn--Loewner theorem~\ref{Thorn}). In particular, Loewner showed that
for $M < \binom{n}{2}$ the derivative $\Delta^{(M)}(0) = 0$ vanishes;
this can be seen from Theorem~\ref{Phorn} via the pigeonhole principle
and since $s_\bm(\bu) = 0$ if $\bm$ has two equal coordinates. Moreover,
for $M = \binom{n}{2}$, there is a unique partition: $M = 0 + 1 + \cdots
+ (n-1)$, and for it the result has a simpler form since the Schur
polynomial factor is not manifested: $s_\bm(\bu)^2 = 1$.

\item In Proposition~\ref{Phorn-alg}, we generalize Loewner's computation
even further, to work over any commutative ground ring.
\end{enumerate}

\section{Theorem~\ref{Phorn}: from any smooth function to Schur
polynomials, to symmetric function identities}\label{S2}

In this section we prove Theorem~\ref{Phorn}, and use it to extend
classical symmetric function identities by Cauchy and Frobenius, as well
as modern variants in~\cite{BGKP-fixeddim,KT}, to a unifying master
identity; see Theorem~\ref{Tsymm}.


\begin{proof}[Proof of Theorem~\ref{Phorn}]
Let ${\bf w}_k$ denote the $k$th column of $a {\bf 1}_{n \times n} + t
\bu \bv^T$; thus ${\bf w}_k$ has $j$th entry $a + t u_j v_k$. To
differentiate $\Delta(t)$, we use the multilinearity of the determinant
and the Laplace expansion of $\Delta(t)$ into a linear combination of
$n!$ monomials, each of which is a product of $n$ terms $f(a + t u_j
v_k)$. By the product rule, taking the derivative yields $n$ terms from
each monomial, and one may rearrange all of these terms into $n$
clusters of terms (grouping by the column that gets differentiated),
and regroup using the Laplace expansion to obtain:
\[
\Delta'(t) = \sum_{k=1}^n \det (
f[{\bf w}_1] \ | \ \cdots \ | \
f[{\bf w}_{k-1}] \ | \
v_k \bu \circ f'[{\bf w}_k] \ | \
f[{\bf w}_{k+1}] \ | \ \cdots \ | \
f[{\bf w}_n] ).
\]

Now apply the derivative repeatedly, using this principle. By the Chain
Rule, $\Delta^{(M)}(0)$ is an integer linear combination of terms of the
form
\begin{align}\label{Ederivative}
\begin{aligned}
&\ \det (v_1^{m_0} \bu^{\circ m_0} \circ f^{(m_0)}[a {\bf 1}_{n \times
1}] \ | \ \cdots \ | \ v_n^{m_{n-1}} \bu^{\circ m_{n-1}} \circ
f^{(m_{n-1})}[a {\bf 1}_{n \times 1}])\\
= &\ \det (f^{(m_0)}(a) v_1^{m_0} \bu^{\circ m_0} \ | \ \cdots \ | \
f^{(m_{n-1})}(a) v_n^{m_{n-1}} \bu^{\circ m_{n-1}}),
\end{aligned}
\end{align}
where ${\bf 1}_{n \times 1} = (1,\dots,1)^T \in \R^n$, and 
$m_0 + \cdots + m_{n-1} = M$ with all $m_k \geq 0$.

From each such determinant, one may factor out the product
$\prod_{k=0}^{n-1} f^{(m_k)}(a)$. Now $\Delta^{(M)}(0)$ is obtained by
summing the determinants corresponding to applying $m_0, m_1, \dots,
m_{n-1}$ derivatives to the columns in some order, for all partitions
$\bm = (m_{n-1}, \dots, m_0)$ of $M$. We first compute the integer
multiplicity of each such determinant, noting by symmetry that these
multiplicities are all equal.
As we are applying $M$ derivatives to $\Delta$ (before evaluating at
$0$), the $m_0$ derivatives applied to get $f^{(m_0)}$ in some (fixed)
column can be any of $\binom{M}{m_0}$; now the $m_1$ derivatives applied
to get $f^{(m_1)}$ in a (different) column can be chosen in $\binom{M -
m_0}{m_1}$ ways; and so on. Thus, the multiplicity is
\begin{align*}
\binom{M}{m_0} \binom{M - m_0}{m_1} \cdots \binom{M - m_0 - \cdots -
m_{n-2}}{m_{n-1}} = &\ \frac{M!}{\prod_{k=0}^{n-1} m_k!}\\
= &\ \binom{M}{m_0,
m_1, \dots, m_{n-1}}.
\end{align*}

The next step is to compute the sum of all determinant terms. Each term
corresponds to a unique permutation of the columns $\sigma \in S_n$, with
say $m_{\sigma_k - 1}$ the order of the derivative applied to the $k$th
column $f[{\bf w}_k]$. Using~\eqref{Ederivative}, the determinant
corresponding to $\sigma$ equals
\begin{align*}
&\ \prod_{k=0}^{n-1} f^{(m_k)}(a) v_k^{m_{\sigma_k - 1}} \cdot
(-1)^\sigma \cdot \det( \bu^{\circ m_0} \ | \ \bu^{\circ m_1} \ | \ \cdots \
| \ \bu^{\circ m_{n-1}} )\\
= &\ V(\bu) s_\bm(\bu) \prod_{k=0}^{n-1} f^{(m_k)}(a) \cdot (-1)^\sigma
\prod_{k=0}^{n-1} v_k^{m_{\sigma_k - 1}}.
\end{align*}
Summing this term over all $\sigma \in S_n$ yields:
\begin{align*}
&\ V(\bu) s_\bm(\bu) \prod_{k=0}^{n-1} f^{(m_k)}(a) \sum_{\sigma \in S_n}
(-1)^\sigma \prod_{k=0}^{n-1} v_k^{m_{\sigma_k - 1}}\\
= &\ V(\bu) s_\bm(\bu) \prod_{k=0}^{n-1} f^{(m_k)}(a) \cdot \det (
\bv^{\circ m_0} \ | \ \bv^{\circ m_1} \ | \ \cdots \ | \ \bv^{\circ
m_{n-1}} )\\
= &\ V(\bu) s_\bm(\bu) \prod_{k=0}^{n-1} f^{(m_k)}(a) \cdot V(\bv)
s_\bm(\bv).
\end{align*}
Now multiply by the (common) integer multiplicity to complete the proof.
\end{proof} 

\subsection{General power series determinants and Schur polynomials}

We next present a novel application of Theorem~\ref{Phorn}, which unifies
and extends classical and modern determinantal identities. Begin by
recalling an identity for Cauchy's determinant \cite{Cauchy}: if $B$ is
the $n \times n$ matrix with entries $(1 - u_j v_k)^{-1} := \sum_{M \geq
0} (u_j v_k)^M$ for variables $u_j, v_k$ and $1 \leq j, k \leq n$, then
\begin{equation}\label{Ecauchy}
\det B = V(\bu) V(\bv) \sum_\bm s_\bm(\bu) s_\bm(\bv),
\end{equation}

\noindent where the notation was explained in Definition~\ref{Dschur},
and the sum runs over all partitions $\bm$ with at most $n$ parts. See
\cite[Chapter I.4, Example 6]{Macdonald} for a proof. Usually this
is written with infinitely many indeterminates $u_j, v_k$; we work with
$u_1, \dots, u_n; v_1, \dots, v_n$ by specializing the other variables to
zero. See also \cite[Section 5]{Ku} and the references therein, as well
as \cite{IIO, IOTZ, Kr1, Kr2, LLT, Macdonald} for other such identities
involving symmetric functions.

In particular, one can apply to all entries of the matrix $\bu \bv^T$
power series other than $f(x) = 1/(1-x) = \sum_{M \geq 0} x^M$, and then
compute the determinant. For instance, if $f(x)$ has fewer than $n$
monomials then $f[ \bu \bv^T ]$ is a sum of fewer than $n$ rank-one
matrices, so it is singular. The case when $f$ has precisely $n$ or $n+1$
monomials was crucially used in joint works \cite{BGKP-fixeddim,KT}, to
study entrywise positivity preservers. Another explicit formula --
see~\eqref{Eschlosser} -- was shown by Frobenius \cite{Fr} in greater
generality than~\eqref{Ecauchy}. This formula also appears in
Rosengren--Schlosser \cite[Corollary 4.7]{RS}, with $(1-cx)/(1-x)$ in
place of $1/(1-x)$ as in~\eqref{Ecauchy}. In this case, one has:
\begin{align}\label{Eschlosser}
&\ \det \left( \frac{1 - c u_j v_k}{1 - u_j v_k} \right)_{j,k=1}^n\\
= &\ V(\bu) V(\bv) (1-c)^{n-1} \left( \sum_{\bm\; :\; m_0 = 0} s_\bm(\bu)
s_\bm(\bv) + (1-c) \sum_{\bm\; :\; m_0 > 0} s_\bm(\bu) s_\bm(\bv)
\right).
\notag
\end{align}

With this background, we can state the master identity that extends
Equations~\eqref{Ecauchy} and~\eqref{Eschlosser} from $(1-cx)/(1-x)$ for
a scalar/parameter $c$, to all power series -- including arbitrary
polynomials -- and with an additional $\Z^{\geq 0}$-grading:

\begin{theorem}\label{Tsymm}
Fix a commutative unital ring $R$ and let $t$ be an indeterminate. Let
$f(t) := \sum_{M \geq 0} f_M t^M \in R[[t]]$ be an arbitrary formal power
series. Given vectors $\bu, \bv \in R^n$ for some $n \geq 1$, and with
the notation in Definition~\ref{Dschur}, we have:
\begin{equation}\label{Esymm}
\det f[ t \bu \bv^T ] = V(\bu) V(\bv) \sum_{M \geq \binom{n}{2}} t^M
\sum_{\bm = (m_{n-1}, \dots, m_0) \; \vdash M} s_\bm(\bu) s_\bm(\bv)
\prod_{k=0}^{n-1} f_{m_k}.
\end{equation}
\end{theorem}

\noindent All of the aforementioned examples are special cases of the
$t=1$ case of Theorem~\ref{Tsymm}.

Multiple approaches can be used to show Theorem~\ref{Tsymm}; we present
one approach and outline another.
We begin by formulating Theorem~\ref{Phorn} in greater generality,
algebraically. Fix a commutative (unital) ring $R$ and an $R$-algebra
$S$. The first step is to formalize the notion of the derivative, on a
sub-class of $S$-valued functions. This is more than just the commonly
used notion of a \textit{derivation}, so we give it a different name.

\begin{definition}\label{Ddiffcalc}
Given a commutative unital ring $R$, a commutative $R$-algebra $S$ (with
$R \subset S$), and an $R$-module $X$, a \textit{differential calculus}
is a pair $(A, \partial)$, where $A$ is an $R$-subalgebra of functions $:
X \to S$ (under pointwise addition and multiplication and $R$-action)
that contains the constant functions, and $\partial : A \to A$ satisfies
the following properties:
\begin{enumerate}
\item $\partial$ is $R$-linear:
\[
\partial \sum_j r_j f_j = \sum_j r_j \partial f_j, \qquad \forall r_j \in
R, \ f_j \in A, \ \forall j.
\]

\item $\partial$ is a derivation (product rule):
\[
\partial(fg) = f \cdot (\partial g) + (\partial f) \cdot g, \qquad
\forall f,g \in A.
\]

\item $\partial$ satisfies a variant of the Chain Rule for composing with
linear functions: if $x' \in X, r \in R$, and $f \in A$, then the
function $g : X \to S, \ g(x) := f(x' + rx)$ also lies in $A$, and
moreover,
\[
(\partial g)(x) = r \cdot (\partial f)(x' + rx).
\]
\end{enumerate}
\end{definition}

For example, the algebra of smooth functions from the real line to itself
is a differential calculus, with $R = S = X = \R$ and $\partial = d/dx$.

We can now state an algebraic generalization of Loewner's calculations.
The proof is essentially the same as for Theorem~\ref{Phorn}, and is
hence omitted.

\begin{prop}\label{Phorn-alg}
Suppose $R, S, X$ are as in Definition~\ref{Ddiffcalc}, with an
associated differential calculus $(A,\partial)$. Fix an integer $n>0$,
two vectors $\bu, \bv \in R^n$, a vector $a \in X$, and a function $f \in
A$; and $\Delta : X \to R$ via:
\[
\Delta(t) := \det f[ a {\bf 1}_{n \times n} + t \bu \bv^T], \quad t \in
X.
\]
Then,
\[
(\partial^M \Delta)(0_X) = \sum_{\bm \vdash M} \binom{M}{m_0, m_1, \dots,
m_{n-1}} V(\bu) V(\bv) s_\bm(\bu) s_\bm(\bv) \prod_{k=0}^{n-1}
(\partial^{m_k} f)(a),
\]
where we sum over all partitions $\bm = (m_{n-1}, \dots, m_0)$ of $M$. In
particular,
\[
\Delta(0_X) = (\partial \Delta)(0_X) = \cdots =
(\partial^{\binom{n}{2}-1} \Delta)(0_X) = 0_R.
\]
\end{prop}

The algebra $A$ is supposed to remind the reader of smooth functions. One
can instead work with an appropriate algebraic notion of $M$-times
differentiable functions in order to generalize Theorem~\ref{Phorn} to a
finite degree of differentiability.

Proposition~\ref{Phorn-alg} helps provide one approach to proving the
master identity:

\begin{proof}[Proof of Theorem~\ref{Tsymm}]
The idea is to apply Proposition~\ref{Phorn-alg} to the differential
calculus
\[
X := t \; R[[t]], \quad S = A := R[[t]],
\]
where $f(t) \in A$ acts on $g(t) \in X$ by composition: $f(g(t))$. We set
$a=0$ here, and the composition converges in the $t$-adic topology by
choice of $X$. Since $a=0$, we have $\det f[ t \bu \bv^T ] = \Delta(t)$.
The problem in proceeding thus is that one needs to clear denominators
and work with $\Delta^{(M)}(0_R) / M!$ and $f_{m_k} = f^{(m_k)}(0_R) /
m_k!$, the Maclaurin coefficients of $\Delta$ and $f$ respectively; but
to work with these requires $R$ to have characteristic zero.

Thus, we begin by observing that the identity~\eqref{Esymm} is of a
universal nature: if it holds for the polynomial ring
\[
R = \Z [ X_1, \dots, X_n, \ Y_1, \dots, Y_n, \ Z_0, Z_1, \dots ]
\]
with $u_j = X_j, \ v_k = Y_k, \ f_m = Z_m\ \forall j,k,m$ algebraically
independent elements, then one may specialize to any given ground ring
$R$ -- or more precisely, to the subring of $R$ generated by $1, \ u_1,
\dots, u_n, \ v_1, \dots, v_n, \ f_0, f_1, \dots$. Hence we assume in the
rest of the proof that
\[
R = \Z[ u_1, \dots, u_n, \ v_1, \dots, v_n , \ f_0, f_1, \dots],
\]
with $u_j, v_k, f_m$ being algebraically independent.

We first work over a slightly larger ring:
\[
R' := R \otimes_\Z \mathbb{Q} = \mathbb{Q}[ \{ u_j, v_j : 1 \leq j \leq
n, \ f_m : m \geq 0 \}].
\]
In this setting, apply Proposition~\ref{Phorn-alg} with $X := t \;
R'[[t]]$, $S = A := R' [[ t ]]$, and $a := 0_{R'}$, and define $\partial
: A \to A$ to be the usual derivative:
\[
\partial \sum_{M \geq 0} g_M t^M := \sum_{M \geq 1} M g_M t^{M-1} \qquad
(g_M \in R,\ \forall M \geq 0).
\]
One verifies that $(A, \partial)$ is a differential calculus for the data
$(R',S,X)$.

We now prove the result over $R'$. Notice that $\Delta(t) = \det f[ t \bu
\bv^T ]$ is a linear combination of finite products of elements of
$R'[[t]]$, hence lies in $R'[[t]]$. If $\Delta(t) = \sum_{M \geq 0}
\delta_M t^M$, then by Proposition~\ref{Phorn-alg} one can compute each
of its Maclaurin coefficients $\delta_M \in R'$ via: $\delta_M =
\frac{(\partial^M \Delta)(0_{R'})}{M!}$, and hence each of its
Taylor--Maclaurin polynomials as well. Taking limits in the $t$-adic
topology, and recalling from Proposition~\ref{Phorn-alg} that $\delta_M =
0_{R'}$ for $M < \binom{n}{2}$, we compute:
\begin{align*}
\Delta(t) = &\ \sum_{M \geq \binom{n}{2}} t^M \frac{(\partial^M
\Delta)(0_{R'})}{M!} \\
= &\ \sum_{M \geq \binom{n}{2}} t^M \sum_\bm V(\bu) V(\bv) s_\bm(\bu)
s_\bm(\bv) \prod_{k=0}^{n-1} \frac{(\partial^{m_k} f)(0_{R'})}{m_k!},
\end{align*}
and this concludes the proof for $R' = \mathbb{Q}[u_1, \dots, u_n, \ v_1,
\dots, v_n, \ f_0, f_1, \dots]$.

While we just showed the identity~\eqref{Esymm} over $R'$, here both
sides of~\eqref{Esymm} belong to $R[[t]]$, where $R = \Z[\{ u_j, v_j : 1
\leq j \leq n, \ f_m : m \geq 0 \}]$. By our discussion on universality,
the result follows for a general commutative unital ring.
\end{proof}

We also sketch an alternative approach to proving Theorem~\ref{Tsymm},
via matrix calculus. In the $t$-adic topology, $f(t) = \lim_{M \to
\infty} f_{\leq M}(t)$, where $f_{\leq M}(t)$ is the $M$th
Taylor--Maclaurin polynomial of $f$ for $M \geq 0$, given by $f_{\leq
M}(t) := \sum_{m=0}^M f_m t^m$. But for $f_{\leq M}$ we have an explicit
matrix factorization:
\begin{align*}
&\ f_{\leq M} [ t \bu \bv^T ]\\
= &\ \begin{pmatrix}
1 & u_1 & \cdots & u_1^M \\
1 & u_2 & \cdots & u_2^M \\
\vdots & \vdots& \ddots & \vdots \\
1 & u_n & \cdots & u_n^M
\end{pmatrix}
\cdot
\begin{pmatrix}
f_0 & 0 & \cdots & 0 \\
0 & f_1 t & \cdots & 0 \\
\vdots & \vdots & \ddots & \vdots \\
0 & 0 & \cdots & f_M t^M
\end{pmatrix}
\cdot
\begin{pmatrix}
1 & v_1 & \cdots & v_1^M \\
1 & v_2 & \cdots & v_2^M \\
\vdots & \vdots& \ddots & \vdots \\
1 & v_n & \cdots & v_n^M
\end{pmatrix}^T
\end{align*}
and hence one can compute $\det f_{\leq M}[ t \bu \bv^T ]$ via the
Cauchy--Binet formula. Now take the $t$-adic limit and rearrange terms to
deduce via the $t$-adic continuity of the determinant function:
\begin{equation}\label{Ecauchybinet}
\det f[t \bu \bv^T] = V(\bu) V(\bv) \sum_{0 \leq m_0 < m_1 < \cdots <
m_{n-1}} s_\bm(\bu) s_\bm(\bv) \prod_{k=0}^{n-1} f_{m_k} t^{m_k},
\end{equation}
where $\bm = (m_{n-1}, \dots, m_0)$. But this sum equals the right-hand
side of~\eqref{Esymm}.

We conclude this section by returning to the real topology and working
again over $\R$. As Theorem~\ref{Tsymm} holds in the $t$-adic topology,
it is natural to ask about the convergence of the series~\eqref{Esymm} as
a real function. This indeed holds, as a consequence of
Theorem~\ref{Phorn}:

\begin{cor}\label{Creal}
Fix scalars $\epsilon > 0$ and $a \in \R$, and vectors $\bu, \bv \in
\R^n$ for some integer $n>0$. If $f : [a, a+\epsilon) \to \R$ has the
power series expansion
\[
f(x) = \sum_{M \geq 0} f_M (x-a)^M, \qquad x \in [a, a + \epsilon),
\ f_M \in \R,
\]
and we define $\Delta(t) := \det f[ a {\bf 1}_{n \times n} + t \bu \bv^T
]$ for sufficiently small $t$, then $\Delta(t)$ equals the right-hand
side of Equation~\eqref{Esymm} for sufficiently small $t$.
\end{cor}

\begin{proof}
If $f$ has a power series expansion around/near $a$, then so does
$\Delta$ near $0$, since it is a linear combination of finite products of
$f$-values near $a$. Thus $\Delta$ (real analytic near $0$) can be
recovered from its Maclaurin coefficients by repeating the same
computation as in the proof of Theorem~\ref{Tsymm}. The Maclaurin
coefficients of $\Delta$ are computed in Theorem~\ref{Phorn}.
\end{proof}

\section{Horn--Loewner master theorem~\ref{Tmain}: additional remarks,
and proof}\label{Sfurther}

Finally, we return to the Horn--Loewner theorem~\ref{Thorn} and its
variants. The goal of this section is to prove our Master
Theorem~\ref{Tmain}; however, we begin by filling in the details omitted
in the Introduction.

One reason why the Horn--Loewner theorem~\ref{Thorn} (and its variants)
was significant was because it is sharp in the number of non-negative
derivatives. We now provide several examples of this phenomenon.

\begin{example}\label{Exbddunbdd}\hfill
\begin{enumerate}
\item If one restricts to the class of power functions $f(x) = x^\alpha$
-- with $\alpha$ possibly non-integral -- then FitzGerald and Horn
\cite{FitzHorn} showed that such entrywise preservers of positivity on
$\bp_n((0,\infty))$ correspond precisely to $\alpha \in \Z^{\geq 0} \cup
[n-2,\infty)$ -- note that the case of $\alpha \in \Z^{\geq 0}$ follows
immediately from the Schur product theorem \cite{Schur1911}. Thus if
$\alpha \in (n-2, n-1)$, then $f(x) = x^\alpha$ satisfies
Theorem~\ref{Thorn}; moreover, $f^{(n)}$ is negative on $I = (0,\infty)$,
showing that Theorem~\ref{Thorn} is sharp.

\item Let $I = (-\rho, \rho)$ or $[0,\rho)$ with $0 < \rho < \infty$. If
one restricts to polynomial functions $f(x) = \sum_{k \geq 0} c_k x^k$
acting on $I$, then in~\cite{BGKP-fixeddim} we showed that for any
scalars $c_0, \dots, c_{n-1} > 0$, there exists $c_n < 0$ such that $f(x)
= \sum_{k=0}^n c_k x^k$ preserves positivity on $\bp_n(I)$; moreover,
$f^{(n)}(0) = n! \; c_n < 0$, hence $f^{(n)}(x) < 0$ for $x>0$ small.
This has two consequences: first, it produces the first examples of
polynomials / power series with a negative coefficient that preserve
positivity in a fixed dimension. Moreover, this produces polynomial
functions (the previous example produces power functions) that preserve
positivity in dimension $n$ but not $n+1$.

\item Suppose that $I = (0,\infty)$. In~\cite{KT}, we constructed
polynomials of the special form $\sum_{k=0}^{2n} c_k x^k$ with $c_n < 0 <
c_k$ for all $k \neq n$, which entrywise preserve positivity on
$\bp_n((0,\infty))$. These are the first examples with negative
coefficients, and in particular, the first polynomial examples that work
over $\bp_n(I)$ but not over $\bp_{n+1}(I)$. The following proposition
provides an explicit family of such polynomials -- in fact, the simplest
such family -- when $n=2$. This is followed by a remark about all
polynomials of the special form above when $n=2$; the case of general $n
\geq 2$ is in~\cite{KT}.
\end{enumerate}
\end{example}

\begin{prop}\label{Pexample}
For $\varepsilon > 0$, define $f_\varepsilon(x) = 1 + x - \varepsilon x^2
+ x^3 + x^4$. If $0 < \varepsilon \leq \frac{1}{10}$, then
$f_\varepsilon[-]$ preserves positivity on $\bp_2([0,\infty))$.
\end{prop}

In particular, such entrywise polynomials $f_\varepsilon[-]$ preserve
positivity on $\bp_2((0,\infty))$, but not on $\bp_3((0,\rho))$ for any
$0 < \rho \leq \infty$, by Lemma~\ref{Lhorn}.

\begin{proof}
We first claim that $f_\varepsilon$ is positive on $(0,\infty)$ if
$\varepsilon \leq 2$. This follows by adding the inequalities $x^2 < 1 +
x^4$ and $x^2 < x + x^3$, for any $x > 0$.

Next suppose that $\bu = (u_1, u_2)^T \in (0,\infty)^2$ is a column and
$A = \bu \bu^T \in \bp_2((0,\infty))$ has rank one. We claim that
$f_\varepsilon[A]$ is positive semidefinite if $0 < \varepsilon \leq
\frac{1}{10}$. By relabelling indices and continuity, it suffices to show
this for $0 < u_1 < u_2$. First note that all entries of
$f_\varepsilon[\bu \bu^T]$ are non-negative, by the preceding paragraph.
Next, we apply~\eqref{Ecauchybinet} -- with $n = 2$, $\bv = \bu$, $t=1$,
and $f = f_\varepsilon$ -- to compute the determinant:
\begin{align}\label{Edet}
&\ \det f_\varepsilon[ \bu \bu^T ]\\
= &\ V(\bu)^2 \left[ \sum_{0 \leq j < k \leq
3} s_{(n_k, n_j)}(\bu)^2 - \varepsilon \sum_{j=0}^1 s_{(M,n_j)}(\bu)^2 -
\varepsilon \sum_{j=2}^3 s_{(n_j,M)}(\bu)^2 \right], \notag
\end{align}
where, for ease of exposition, we define and work with
\begin{equation}
(n_0, n_1, M, n_2, n_3) := (0,1,2,3,4).
\end{equation}

To show that $\det f_\varepsilon[ \bu \bu^T ]$ is non-negative, we employ
bounds on Schur polynomials in two variables. Given integers $0 \leq j <
k$, notice that
\[
s_{(k,j)}(u_1,u_2) = \frac{u_2^k u_1^j - u_2^j u_1^k}{u_2 - u_1} = (u_1
u_2)^j (u_2^{k - j - 1} + u_2^{k - j - 2} u_1 + \cdots + u_1^{k - j -
1}).
\]
If $0 < u_1 < u_2$ as above, then the largest among the $k - j$ monomials
here is $u_1^j u_2^{k - 1}$, so
\begin{equation}\label{Eschurbound}
u_1^j u_2^{k - 1} \leq s_{(k, j)}(u_1, u_2) \leq (k - j) u_1^j u_2^{k -
1}, \qquad \forall 0 < u_1 < u_2, \ \ 0 \leq j < k.
\end{equation}
(See \cite[Proposition~3.1]{KT} for the corresponding bounds on general
Schur polynomials.) In particular, dividing~\eqref{Edet} by $V(\bu)^2 >
0$ and using~\eqref{Eschurbound}, we have:
\[
\frac{\det f_\varepsilon[\bu \bu^T]}{V(\bu)^2} \geq \sum_{0 \leq j < k
\leq 3} u_1^{2n_j} u_2^{2(n_k - 1)} - \varepsilon \sum_{0 \leq j \leq 3}
(M - n_j)^2 u_1^{2\min(n_j, M)} u_2^{2\max(n_j,M)-2}
\]

We claim that the right-hand side is non-negative when $0 < \varepsilon
\leq \frac{1}{10}$. This is equivalent to showing that if $0 < u_1 <
u_2$, then
\[
\varepsilon^{-1} \geq 10 = \sum_{j=0}^3 (M - n_j)^2 \geq
\frac{\displaystyle \sum_{0 \leq j \leq 3} (M -
n_j)^2 u_1^{2\min(n_j, M)} u_2^{2\max(n_j,M)-2}}{\displaystyle \sum_{0
\leq j < k \leq 3} u_1^{2n_j} u_2^{2(n_k - 1)}}.
\]
For this, it suffices to show that each monomial $u_1^{2\min(n_j, M)}
u_2^{2\max(n_j,M)-2}$ in the numerator is bounded above by the sum in the
denominator. We claim more strongly that each such monomial is bounded
above by one summand in the denominator. This is because for each $j$,
there exist indices $i \neq k$ in $\{ 0, 1, 2, 3 \} \setminus \{ j \}$
such that $n_i < M < n_k$. Now for instance if $j=1$, then
\[
u_1^{2 n_1} u_2^{2M - 2} \leq \quad \begin{cases}
u_1^{2 n_0} u_2^{2(n_1 - 1)}, & \text{if } u_2 < 1,\\
u_1^{2 n_1} u_2^{2(n_3 - 1)}, & \text{if } u_2 \geq 1,
\end{cases}
\]
and similarly for the other three summands.

This shows the result for all rank-one matrices in $\bp_2((0,\infty))$.
For the general case, we use the extension principle
\cite[Theorem~3.5]{KT}, which says that if $f : (0,\infty) \to \R$ is a
continuously differentiable function such that $f[-]$ preserves
positivity on rank-one matrices in $\bp_n((0,\infty))$ and $f'[-]$
preserves positivity on $\bp_{n-1}((0,\infty))$ for some $n \geq 2$, then
$f[-]$ preserves positivity on $\bp_n((0,\infty))$. Via this result, it
suffices to show that $f'_\varepsilon[-]$ preserves positivity on
$\bp_1((0,\infty))$ if $0 < \varepsilon \leq \frac{1}{10}$. But this
follows by noting that
\[
f'_\varepsilon[(x)_{1 \times 1}] = f'_\varepsilon(x) \geq f'_{1/10}(x) >
(1 - \frac{x}{10})^2 \geq 0,\qquad \forall x > 0.
\]
Thus if $0 < \varepsilon \leq \frac{1}{10}$, then $f_\varepsilon[-]$
preserves positivity on $\bp_2((0,\infty))$, so on $\bp_2([0,\infty))$ by
continuity.
\end{proof}

\begin{remark}
In this remark, we consider the more general family of polynomials
\[
f_\varepsilon(x) = c_{n_0} x^{n_0} + c_{n_1} x^{n_1} - \varepsilon x^M +
c_{n_2} x^{n_2} + c_{n_3} x^{n_3}, \qquad \varepsilon > 0
\]
where $0 \leq n_0 < n_1 < M < n_2 < n_3$ are integers, and $c_{n_j} \in
(0,\infty)$ are fixed scalars. If
\begin{equation}
0 < \varepsilon \leq \varepsilon_0 := \min \{ c_{n_j} : 0 \leq j \leq 3
\} \cdot \left( \sum_{j=0}^3 (M - n_j)^2 \right)^{-1}
\end{equation}
then we claim that $f_\varepsilon[-]$ preserves positivity on
$\bp_2([0,\infty))$. Here we outline the argument when it differs from
the proof of Proposition~\ref{Pexample}, mainly due to the integers $n_j
\geq 0$ being arbitrary.

The first reduction is to the case when all $c_{n_j} = 1$. Note that for
any positive semidefinite matrix $A_{n \times n}$,
one has $f_\varepsilon[A] - g_\varepsilon[A] \in \bp_n$,
where $g_\varepsilon(x) := \min_j c_{n_j} \cdot \sum_{j=0}^3 x^{n_j} -
\varepsilon x^M$. This is because $A^{\circ n_j} \in \bp_n$ for all $j$,
by the Schur product theorem \cite{Schur1911}. Thus if the result holds
for $(\min_j c_{n_j})^{-1} g_\varepsilon(x)$ -- with the bound
$\varepsilon^{-1} \geq \sum_{j=0}^3 (M - n_j)^2$ -- then the result holds
for $f_\varepsilon$ as desired.

Thus, suppose that $c_{n_j} = 1\ \forall j$ and $f_\varepsilon(x) =
x^{n_0} + x^{n_1} - \varepsilon x^M + x^{n_2} + x^{n_3}$, akin to its
special case in Proposition~\ref{Pexample}. Now $\varepsilon_0 \leq
\frac{1}{10}$, so the same arguments as in the preceding proof show:
(i)~$f_\varepsilon(x) > 0$ if $0 < \varepsilon \leq 2$ and $x>0$; and
(ii)~$f_\varepsilon[-]$ preserves positivity on rank-one matrices in
$\bp_2((0,\infty))$, if $0 < \varepsilon \leq \varepsilon_0$.
The final step -- as in the preceding proof -- is to use the extension
principle, and here we claim the stronger assertion:
$h_{\varepsilon_0}(x) := n_1 x^{n_1 - 1} - M \varepsilon_0 x^{M-1} + n_2
x^{n_2 - 1} \geq 0\ \forall x > 0$.
To prove this, it suffices to multiply by $x>0$ and show that $n_1
(x^{n_1} + x^{n_2}) \geq M \varepsilon_0 x^M$. This would follow from $M
\varepsilon_0 \leq n_1$, which in turn follows from the sub-claim that
\[
2 (M - n_1)^2 \geq M / n_1, \quad \text{where} \quad M \geq n_1 + 1, \ n_1
\geq 1.
\]
This sub-claim is proved by considering the function $h(x) := 2(x -
n_1)^2 - n_1^{-1} x$ on $[n_1,\infty)$. Note that $h'(x) \geq 0$ for $x
\geq n_1 + 1/(4 n_1)$, so $h(M) \geq h(n_1 + 1) = 1 - n_1^{-1} \geq 0$.
\qed
\end{remark}

Our next observation explains why we call Theorem~\ref{Tmain} a master
theorem: it encompasses all of the previously known versions.

\begin{prop}\label{Pvariants}
Theorem~\ref{Tmain} specializes to all results stated prior to it, for
$f$ smooth.
\end{prop}

\begin{proof}
We first show how Theorem~\ref{Tmain} implies Theorem~\ref{Thankel}
(which in turn implies Theorem~\ref{Thorn}). Since $f$ is smooth, by the
discussion following Theorem~\ref{Thankel} we may disregard the
hypothesis concerning $\bp_2(I)$. Choose any $a \in I = (0,\rho)$, and
set
\[
\epsilon := \rho - a, \qquad p = q := n,
\qquad u_k := u_0^{k-1}\ \forall k \in [1,n].
\]
Theorem~\ref{Thankel} now follows from Theorem~\ref{Tmain}, for smooth
$f$.

Next we show how Theorem~\ref{Tmain} implies a stronger version of
Lemma~\ref{Lhorn}(1) -- for smooth functions, not merely power series.
Set $a=0$ and suppose that $f$ has $N \leq \infty$ nonzero derivatives at
$a=0$. Let $l := \min(n, N)$ and denote the smallest $l$ of these orders
of derivatives by $m_0, \dots, m_{l-1}$. Now set
$\epsilon := \rho, \ p := 0, \ q := l$.
It follows by Theorem~\ref{Tmain} that the first $l$ nonzero Maclaurin
coefficients of $f(x)$ at $x=0$ are non-negative, hence positive as
desired. This shows the result -- and with a smaller test set used here
than in Lemma~\ref{Lhorn}(1).

Finally, we show how Lemma~\ref{Lhorn}(2) follows from
Theorem~\ref{Tmain}. By Lemma~\ref{Lhorn}(1), it suffices to consider
only the coefficients of degree $>m'$. Thus, suppose that the assumptions
of Lemma~\ref{Lhorn}(2) hold, and yet $c_{m'} < 0$ is not followed by $n$
positive coefficients of higher degree. First, if $c_{m'}$ is followed by
infinitely many negative coefficients of higher degree, then $f(x) < 0$
for $x \gg 0$. But then $f[ x {\bf 1}_{n \times n}] \not\in \bp_n(\R)$,
contradicting the hypotheses.

Thus $c_{m'}$ is followed by finitely many nonzero higher-order
coefficients. Without loss of generality, we may redefine $m'$ to be the
highest degree coefficient that is negative; thus,
\[
f(x) = \sum_{j=0}^d c_{n_j} x^{n_j} + c_{m'} x^{m'} + \sum_{k=0}^{l-1}
c_{m_k} x^{m_k},
\]
where $0 \leq n_0 < \cdots < n_d < m' < m_{l-1} < \cdots < m_0$ are
integers. Now define
\[
g(x) := \sum_{k=0}^{l-1} c_{m_k} c^{m_0 - m_k} + c_{m'} x^{m_0 - m'} +
\sum_{j=0}^d c_{n_j} x^{m_0 - n_j}.
\]
In other words, $g(x) = x^{m_0} f(1/x)$ for $x>0$. We claim that $g(x)$
entrywise preserves positivity on rank-one matrices $\bu \bu^T \in
\bp_n((0,\infty))$. Indeed,
\[
g[ \bu \bu^T ] = (\bu \bu^T)^{\circ m_0} \circ f[ (\bu \bu^T)^{\circ -1}]
= [ \bu^{\circ m_0} (\bu^{\circ m_0})^T ] \circ f[ \bu^{\circ -1}
(\bu^{\circ -1})^T],
\]
and this is positive semidefinite by the Schur product theorem and since
$\bu^{\circ -1} \in (0,\infty)^n$ as well. But this reduces us to the
previous case of Lemma~\ref{Lhorn}(1), which follows from
Theorem~\ref{Tmain} and implies that $g$ has at least $n$ positive
coefficients of degree lower than $x^{m_0 - m'}$. Therefore $l \geq n$,
contradicting the assumption.
\end{proof}

\begin{remark}\label{Rmollifier}
For completeness, we briefly discuss what happens if one tries to weaken
the smoothness hypothesis in Theorem~\ref{Tmain}. The way that
Horn/Loewner originally proved Theorem~\ref{Thorn} was to appeal to a
result of Boas and Widder \cite{Boas-Widder} by using mollifiers, that
is, convolving $f$ with $\phi(x/\delta)$ for $\delta>0$ and a certain
smooth function $\phi : (-1,0) \to (0, \infty)$. We now explain why it is
not possible to repeat this argument for Theorem~\ref{Tmain} outside of
the setting of the Horn--Loewner setting $p=n$.

Indeed, suppose $p < n$, which we may take to mean $m_p > p < n$. To
repeat the mollifier argument would at least involve changing the
hypothesis 
\[
f[a {\bf 1}_{n \times n} + t \bu \bu^T] \in \bp_n(\R), \ \ \forall t
\]
from each fixed $a$, to all $a$ belonging to an interval $J'$. But now if
we want $f_\delta^{(p)}(a) = 0$, then assuming that $f$ is nice enough
(e.g., if $f, f', \dots, f^{(p)}$ are bounded on $J'$), we compute:
\[
f^{(p)}_\delta(a) = \int_{-\delta}^0 f^{(p)}(a-u) \phi \left(
\frac{u}{\delta} \right) = 0.
\]
From this and since $f^{(p)} \geq 0$ on $I$ for $p<n$ (by
Theorem~\ref{Thorn}), it follows that $f^{(p)}$ vanishes on some interval
$J$ containing $a$, hence so does $f^{(r)}$ for all $r \geq p$. But this
does not reconcile with $f^{(m_p)}(a) \neq 0$ for $m_p > p$.
\end{remark}

Given Theorem~\ref{Tmain} and the discussion in Remark~\ref{Rmollifier},
we next observe that the original non-pointwise Theorem~\ref{Thorn}, as
well as its strengthening in Theorem~\ref{Thankel}, admit a small
generalization: \textit{the domain need not begin at $0$}.

\begin{cor}
Theorems~\ref{Thorn} and~\ref{Thankel} hold for every open subinterval $I
\subset (0,\infty)$.
\end{cor}

This result is \textit{equivalent} to the formulation for $I =
(0,\infty)$, and this equivalence is immediate. Indeed, if $I = (r,s)$
with $0 < r < s \leq \infty)$, then one works instead with the function
$g : (0, s-r) \to \R, \ g(x) := f(x+r)$, and this reduces the result to
Theorems~\ref{Thorn} and~\ref{Thankel} respectively.

Our final remark, before proving the Master Theorem~\ref{Tmain},
addresses a classical result by Boas and Widder:

\begin{remark}\label{RBWerratum}
As discussed in \cite{BGKP-hankel,GKR-lowrank, horn}, the Horn--Loewner
theorem for general functions -- on all open intervals $(0,\rho)$ for $0
< \rho \leq \infty$ -- in fact follows from its smooth version, using
mollifiers and a 1940 result by Boas and Widder. Here we record a minor
typo in Boas and Widder's proof of their (rather remarkable!) main result
in~\cite{Boas-Widder}.
The authors begin the proof of \cite[Lemma 13]{Boas-Widder} by claiming
that if $I \subset \R$ is an open interval and $f : I \to \R$ is
continuous and has non-negative forward differences of order $k \geq 3$,
then $f'$ is monotonic. However, this is not true as stated: for any $k
\geq 3$, the function $f(x) = x^3$ satisfies these hypotheses on $I =
(-1,1)$, but $f'$ is not monotone on $I$.

We now explain how to fix this issue. One has to claim instead that $f'$
is \textit{piecewise} monotone on $I$. This claim follows by applying in
turn \cite[Lemmas 9, 4, and 11]{Boas-Widder}. The piecewise monotonicity
then suffices to imply the existence of $f'(x\pm)$ at every point in $I$,
and the remainder of the proof of Lemma $13$ in \cite{Boas-Widder} goes
through verbatim.
\end{remark}

\subsection{Proof of the Master Theorem~\ref{Tmain}}

We conclude the paper by using Theorem~\ref{Phorn} to show the
Horn--Loewner Master Theorem~\ref{Tmain}. The following definition, which
may seem somewhat opaque at first glance, will feature in the proof.

\begin{definition}\label{Dadmissible}
Let $a \in \R$ and $\epsilon \in (0,\infty)$. Define $I := [a,
a+\epsilon)$ and suppose $f : I \to \R$ is smooth. We say that a tuple of
integers
\[
0 \leq m_0 < \cdots < m_{n-1}
\]
is \textit{admissible} for these data if, for all tuples $(l_0, \dots,
l_{n-1})$ of non-negative integers such that $\sum_k l_k \leq \sum_k
m_k$, at least one of the following three possibilities holds:
\begin{enumerate}
\item The $l_k$ are not pairwise distinct.
\item There exists $k$ such that $f^{(l_k)}(a) = 0$.
\item $\{ l_0, \dots, l_{n-1} \} = \{ m_0, \dots, m_{n-1} \}$.
\end{enumerate}
\end{definition}

\noindent Notice that this definition is independent of $\epsilon > 0$.
We now characterize all admissible tuples of each given length:

\begin{lemma}\label{Ladmissible}
Given $a \in \R,\ \epsilon>0$, an integer $n>0$, and $f : [a, a+\epsilon)
\to \R$ smooth, an integer tuple $0 \leq l_0 < \cdots < l_{n-1}$ is
admissible for these data if and only if:
\begin{enumerate}
\item Either $f$ has at most $n-1$ nonzero derivatives at $a$; or

\item If the integers $0 \leq m_0 < \cdots < m_{n-1}$ denote the $n$
lowest-order nonzero derivatives of $f(x)$ at $x=a$, then either $l_k =
m_k\ \forall k$ or $\sum_k l_k < \sum_k m_k$.
\end{enumerate}
In particular, given $a,\epsilon,f,n$, there are either finitely many
length $n$ admissible tuples, or every length $n$ tuple of pairwise
distinct non-negative integers is admissible.
\end{lemma}

The tuple $m_k = k, \ 0 \leq k < n$ was used in Loewner's determinant
computation and proof of Theorem~\ref{Thorn}, and this tuple is easily
seen to be admissible. As discussed in the discussion following
Remark~\ref{Rhorn}, in this special case the argument is somewhat less
involved and the underlying use of admissibility is not revealed; but
this subtlety is made clear in the proof of Theorem~\ref{Tmain}.

\begin{proof}
Clearly if $f$ has at most $n-1$ nonzero derivatives at $a$, then every
integer tuple $0 \leq l_0 < \cdots < l_{n-1}$ forms an admissible
tuple, by the pigeonhole principle.

Now suppose that $f^{(m_k)}(a) \neq 0$ for all $0 \leq k \leq n-1$, and
the $m_k$ are minimal with this property, as well as pairwise distinct.
One checks that the $m_k$ form an admissible tuple. For any tuple
$(m'_k)$ such that $\sum_k m'_k \geq \sum_k m_k$ but $\{ m'_0, \dots,
m'_{n-1} \} \neq \{ m_0, \dots, m_{n-1} \}$, we can choose $l_k := m_k$
in Definition~\ref{Dadmissible} to verify that $(m'_k)$ is not
admissible.

Finally, if $\sum_k m'_k < \sum_k m_k$, then we claim that $(m'_k)$ is an
admissible tuple. Indeed, choose any tuple $(l_k)$ of pairwise distinct
integers $l_k \geq 0$ with $\sum_k l_k \leq \sum_k m'_k < \sum_k m_k$.
Then conditions (1), (3) in Definition~\ref{Dadmissible} fail to hold, so
condition (2) holds by the minimality of the $m_k$.
\end{proof}

Finally, we have:

\begin{proof}[Proof of Theorem~\ref{Tmain}]
At the outset, set 
\[
m_0 := 0, \quad \dots, \quad m_{p-1} := p-1.
\]
Notice that it suffices to show that $f^{(m_k)}(a) \geq 0$ for $0 \leq k
\leq q-1$. The remaining derivatives at $x=a$ of $f(x)$ of order $\leq
m_{q-1}$ lie in $[p, m_{q-1}] \setminus \{ m_p, \dots, m_{q-1} \}$, and
hence are zero by the choice of the $m_k$.

The second observation is that the given test set of $n \times n$
matrices contains as principal submatrices a corresponding test set of $q
\times q$ matrices. Hence we may restrict to the leading principal $q
\times q$ submatrices of the given test set, and work with only this
reduced test set. In other words, we may assume without loss of
generality that $q=n$.

Having made these reductions, we prove the result. For each $0 \leq
\delta$ small enough, define $f_\delta(x) := f(x) + \delta x^{p-1}$ with
$x \in I$. For the data $a \geq 0$, any $\epsilon > 0$, and $f_\delta$
with any $\delta > 0$, note by Lemma~\ref{Ladmissible}(2) that the tuple
$(m_k)$ is indeed admissible, since the $m_k$ denote the orders of the
first $n$ nonzero derivatives\footnote{In the original proof in
\cite{horn} for $p=q=n$, Horn/Loewner use $f_\delta(x) := f(x) + \delta
x^n$; but for their purposes they could just as well have used any power
$\geq n-1$. As the present proof reveals, in order to examine the
coefficients of nonzero derivatives of order up to $n-1$, the optimal
power to use in the original Horn--Loewner setting would be $n-1$.} of
$f_\delta(x)$ at $x=a$.

Now given $a,t$ and the vector $\bu$ as in the theorem, define
$\Delta(t) := \det f_\delta[a {\bf 1}_{n \times n} + t \bu \bu^T]$ as in
Theorem~\ref{Phorn} (i.e., replacing $f, \bv$ by $f_\delta, \bu$
respectively). Then $\Delta(t) \geq 0$ for $t > 0$, by the hypotheses and
using the Schur product theorem for $x^{p-1}$. From this we obtain:
\[
0 \leq \lim_{t \to 0^+} \frac{\Delta(t)}{t^M}, \qquad \text{where } M :=
m_0 + \cdots + m_{n-1}
\]

\noindent provided this limit exists.

We now claim that $\Delta(0) = \Delta'(0) = \cdots = \Delta^{(M-1)}(0) =
0$, so that one can compute the limit using L'H\^{o}pital's rule.
Indeed, going through the proof of Theorem~\ref{Phorn}, if we choose any
tuple $(l_k)$ of non-negative integers, the
determinant~\eqref{Ederivative} vanishes if any two $l_k$ are equal, or
if $f^{(l_k)}(a) = 0$. Thus if $\sum_k l_k \leq \sum_k m_k$, then the
admissibility of the tuple $(m_k)$ (shown above) implies that $\{ l_0,
\dots, l_{n-1} \} = \{ m_0, \dots, m_{n-1} \}$.

In particular, $\Delta^{(L)}(0) = 0$ for all $0 \leq L < M$, by
Theorem~\ref{Phorn}. Continuing the computation,
\[
0 \leq \lim_{t \to 0^+} \frac{\Delta(t)}{t^M} =
\frac{\Delta^{(M)}(0)}{M!} = V(\bu)^2 s_\bm(\bu)^2 \prod_{k=0}^{n-1}
\frac{f_\delta^{(m_k)}(a)}{m_k!},
\]
where the equalities are by L'H\^{o}pital's rule, and by
Theorem~\ref{Phorn} and the admissibility of $\bm$.
In particular, the right-hand side here is non-negative. Since $\bu$ has
distinct coordinates, we can cancel all positive factors to conclude that
\begin{equation}\label{Eprodineq}
\prod_{k=0}^{n-1} f_\delta^{(m_k)}(a) \geq 0.
\end{equation}
Notice that for $n=1$, this proves the inequality, by sending $\delta \to
0^+$.

We now prove the result by induction on $n=q$. For the induction step, we
know that $f^{(m_k)}(a) \geq 0$ for $0 \leq k \leq n-2$, since the given
test set of $n \times n$ matrices contains as leading principal
submatrices a corresponding test set of $(n-1) \times (n-1)$ matrices.
There are now two cases. If $a=0$, then with $\delta = 0$ and $0 \leq k
\leq n-2$, we have $f_\delta^{(m_k)}(0) > 0$  as was just discussed, so
we may divide and obtain $f^{(m_{n-1})}(0) \geq 0$, as desired.

If instead $a>0$ and $\delta > 0$, then
\[
f_\delta^{(m_k)}(a) = f^{(m_k)}(a) + \delta \; p (p - 1) \cdots (p - m_k
+ 1) a^{p - m_k},
\]
and this is positive for $0 \leq k \leq n-2$ by the induction hypothesis.
(Here we consider separately the cases $0 \leq k < p$ and $k \geq p$.)
Hence by~\eqref{Eprodineq},
\[
f_\delta^{(m_{n-1})}(a) = f^{(m_{n-1})}(a) + {\bf 1}_{p=n} \; \delta \cdot
m_{n-1}! \geq 0, \qquad \forall 0 < \delta \ll 1
\]
and it follows that $f^{(m_{n-1})}(a) \geq 0$.
\end{proof}

\section*{Acknowledgments}

The author thanks Arvind Ayyer, Michael Schlosser, and Suvrit Sra for
valuable discussions.
The author would also like to thank Poornendu Kumar and Shubham Rastogi
for helpful discussions involving the Boas--Widder paper
\cite{Boas-Widder}. Furthermore, the author is very grateful to the
anonymous referee for carefully reading the paper and providing truly
detailed feedback, which helped improve the exposition.




\end{document}